\documentclass[12pt,psamsfonts]{amsart}
\usepackage{amsmath}
\usepackage{amsthm}
\usepackage{amssymb}
\usepackage{amscd}
\usepackage{amsfonts, mathrsfs}
\usepackage{amsbsy}
\usepackage{amssymb}
\usepackage{enumerate}
\usepackage{hyperref}

\newtheorem {theorem}{Theorem}[section]
\newtheorem {proposition} [theorem]{Proposition}

\newtheorem {lemma}  [theorem]{Lemma}

\newtheorem{definition}[theorem]{Definition}
\newtheorem{thmx}{Theorem}

\newcommand{\R}{\mathbb{R}}

\newcommand{\N}{\mathbb{N}}
\newcommand{\Z}{\mathbb{Z}}

\renewcommand{\P}{\mathscr{P}}
\renewcommand{\L}{\mathscr{L}}

\newcommand{\ve}{\varepsilon}
\newcommand{\vp}{\varphi}
\newcommand{\g}{\gamma}
\renewcommand{\o}{\omega}
\newcommand{\al}{\alpha}

\newcommand{\s}{s}
\renewcommand{\a}{\alpha}
\newcommand{\h}{h}

\newcommand{\p}{\partial}
\newcommand{\mb}{\mathbf}
\newcommand{\id}{\text{id}}

\newcommand{\sys}[2]{
  \left\{\!
   \begin{array}{l}
    \dot x=#1 \\[2pt] \dot y=#2
   \end{array}
  \right.
}

\newcommand{\THM}{Theorem B in \cite{Grau2011} for even potentials}
\newcommand{\LEM}{Lemma 4.1 in \cite{Grau2011}}

\begin{document}
\title[limit cycles for perturbed pendulum equations]
{On the number of limit cycles\\ for perturbed pendulum equations}
\author{A.~Gasull, A.~Geyer  and F.~Ma\~{n}osas }

\address{Departament de Matem\`{a}tiques,
Universitat Aut\`{o}noma de Barcelona, 08193 Bellaterra, Barcelona, Spain; \newline $\phantom{x\;}$ Faculty of Mathematics,
University of Vienna, \mbox{Oskar-Morgenstern-Platz 1}, 1090 Vienna, Austria;} \email{gasull@mat.uab.cat,
anna.geyer@univie.ac.at,\newline $\phantom{xxxxxxxxxxxxxx}$ manyosas@mat.uab.cat} \date{}



\maketitle


\begin{abstract}
We consider perturbed pendulum-like equations on the cylinder of the form $  \ddot x+\sin(x)= \varepsilon
\sum_{\s=0}^{m}{Q_{n,\s} (x)\, \dot x^{\s}}$ where $Q_{n,\s}$ are trigonometric polynomials of
degree $n$, and  study  the number  of limit cycles that bifurcate from
the periodic orbits of the unperturbed case $\varepsilon=0$ in terms of $m$ and $n$. Our first result gives
upper bounds on the
number of zeros of its associated first order Melnikov function, in both the oscillatory and the
rotary regions. These upper bounds are obtained expressing the corresponding Abelian integrals in
terms of polynomials and the complete elliptic functions of first and second kind. Some further
results give sharp bounds on the number of zeros of these integrals by identifying subfamilies which
are shown to be Chebyshev systems.
\end{abstract}


\setcounter{tocdepth}{2}


\section{Introduction}

The so-called \emph{Hilbert's $16^{th}\,$Problem} was proposed by David Hilbert at the Paris
conference of the International Congress of Mathematicians in $1900$. The problem is to  determine
the upper bound for the number of limit cycles in two-dimensional polynomial vector fields of
degree $d$, and to investigate their relative positions, see~\cite{Ilia,Li}. There is also a
weaker version, the so-called \emph{infinitesimal or tangential Hilbert's $16^{th}\,$Problem},
proposed by Arnold, which can be stated in the following way: let $\o$ be a real $1$-form with
polynomial coefficients of degree at most $d$, and consider a polynomial $H$ of degree $d+1$. A
closed connected component of a level curve of $H=h$, denoted by $\g_h$, is called an \emph{oval}
of $H$. These ovals form continuous families. The infinitesimal Hilbert's $16^{th}\,$Problem then
asks for an upper bound $V(d)$ of the number of real zeros of the Abelian integral
\begin{equation*}
    I(h) = \int_{\g_h} \o.
\end{equation*}
The bound should be uniform with respect to the polynomial $H$, the
family of ovals $\{\g_h\}$ and the form $\o$, i.e. it should only
depend on the degree $d$, cf.~\cite{Ilia, Iliev1999}. The existence
of $V(d)$ goes back to the works of Khovanskii and Varchenko
(\cite{Kh,Var}). Recently an explicit (non realistic) bound  for
$V(d)$ has been given in \cite{Bin} by Binyamini, Novikov and
Yakovenko.

 There is
a beautiful relationship between limit cycles and zeros of Abelian integrals: Consider a small
deformation of a Hamiltonian vector field
\begin{equation*}
    X_{\ve} = X_H + \ve Y,
\end{equation*}
where $X_H=-H_y \p_x + H_x\p_y$, $Y=P \p_x + Q\p_y$ and $\ve>0$ is a small parameter. Denote by
$d(h,\ve)$ the \emph{displacement function of the  Poincar\'{e} map} of $X_{\ve}$ and consider its
power series expansion in $\ve$. The  coefficients in this expansion are called \emph{Melnikov
functions $M_k(h)$}. Therefore, the limit cycles of the vector field correspond to isolated zeros
of the first non-vanishing Melnikov function. A closed expression of the  first Melnikov function
$M_1(h)=I(h)$ was obtained by Pontryagin which is  given by the Abelian integral
\[I(h)=\int_{\g_h} \o,\quad \mbox{with}\quad \o=P \,dy-Q\, dx.\]
 Hence the number of isolated zeros of $I(h)$, counting multiplicity, provide an
upper bound for the number of ovals of $H$ that generate limit cycles of $X_{\ve}$ for $\ve$ close
to zero. The coefficients of $P$ and $Q$ are considered as parameters, and so $I(h)$ splits into a
linear combination $I(h)=\alpha_0 I_0(h)+ \dots + \alpha_{\ell} I_{\ell}(h),$ for some
$\ell\in\N$, where the coefficients $\alpha_k$ depend on initial parameters and $I_k(h)$ are
Abelian integrals with some $\o_k=x^{i_k} y^{j_k} dx$. Therefore, the problem of finding the
maximum number of isolated zeros of $I(h)$ is equivalent to finding an upper bound for the number
of isolated zeros of any function belonging to the vector space generated by $I_j(h)$, $j=0, \dots
\ell.$ This equivalent problem becomes easier when the basis of this vector space is
a \emph{Chebyshev system}, see Section \ref{pb} for details.\\

We are interested in these considerations  because we want to  analyze in terms of
$m$ and $n$ the number of periodic orbits for perturbed pendulum-like equations of the form
\begin{equation}
\label{E-pendu}
    \ddot x+\sin(x)= \ve \sum_{\s=0}^{m}{Q_{n,\s} (x)\, \dot x^{\s}},
\end{equation}
where for each $s$ the functions $Q_{n,s}$ are trigonometric polynomials of degree at most $n$
and $\ve>0$ is a small parameter.
 The planar system  associated to~\eqref{E-pendu} can be viewed as a trigonometric
 perturbation  of the  Hamiltonian  system
\begin{equation}
\label{E-HamSys}
     \sys{y,}{-\sin(x),}
\end{equation}
with total energy
\begin{equation}
\label{E-Hamiltonian}
    H(x,y)= \frac{y^2}{2} +1-\cos(x),
\end{equation}
which in fact can be considered on the cylinder  $ [-\pi,\pi]\times \R$. In other words, we are
interested  in quantifying in terms of $m$ and $n$  the number of limit cycles  that bifurcate
from the closed ovals of the unperturbed pendulum equation $\ddot x+\sin x=0.$ This problem  can
be seen as an extension of the infinitesimal Hilbert's $16^{th}\,$Problem to the trigonometrical
world.

Notice that for $h\in(0,2)$ the levels $\gamma_h=\{(x,y);H(x,y)=h\}$ are ovals surrounding the
origin, while for $h\in (2,\infty)$ the corresponding levels have two connected components which
are again ovals, one of them contained in the region $y>0$ denoted by $\gamma_h^+$, and the other
one contained in the region $y<0$ denoted by $\gamma_h^-.$ The region corresponding to energies
$h\in(0,2)$ is usually called {\it oscillatory region} and we will denote it by $\mathcal{R}^0.$
The regions with energies $h\in(2,\infty)$ and $\pm y>0$ will be denoted  by $\mathcal{R}^\pm$ and
both together form the so-called {\it rotary region}.\\

The analysis of equations of this form  is also motivated by a
number of problems resulting from pendulum-like equations appearing
in the literature. Examples include  the system
\begin{equation*}
    \sys{ y,}{\alpha+\beta y+\gamma y^2+\delta\cos(2\pi x),}
\end{equation*}
where  $\alpha, \beta, \gamma$ and $\delta$  are real parameters,
which was considered in \cite{CheGasLli87}. Another interesting
example with  pendulum-type behaviour  is the equation
\begin{align*}
    \ddot x+\sin( x ) = \ve \dot x \cos (nx),\quad n\in \N,
\end{align*}
considered by Morozov in \cite{Morozov1989}. The author proves that for $\ve>0$ small enough this
system has exactly  $n-1$ hyperbolic limit cycles in $\mathcal{R}^0$, and no limit cycles in
$\mathcal{R}^\pm.$
The proof relies on a representation of the Abelian integrals in
terms of polynomials and the complete elliptic functions of first and second kind.

 A further example of a pendulum-like equation is  the Josephson equation
\begin{align*}
    \ddot x + \sin(x)=\varepsilon[a-\left(1+\gamma\cos(x)\right)\dot x],
\end{align*}
where $\ve>0$ is a small parameter and $a,\gamma\in\R$. This equation  was studied by various
authors \cite{Bel77a,Ken,Sanders1983,Sanders1986} by analyzing the corresponding averaged system
whose right-hand side consists of three Abelian integrals. Instead of  expressing
these integrals in terms of complete elliptic integrals the authors of \cite{Sanders1986} use
techniques from bifurcation theory to find the bifurcation diagram and corresponding phase
portraits on the cylinder for the resulting two-parameter family of vector fields. Realizing  that
the aforementioned  Abelian integrals satisfy a certain Picard-Fuchs equation, they then analyze
the solutions of the resulting Riccati equations.

Another very related problem is the study of the periodic solutions of the perturbed  whirling
pendulum,
\[
\ddot x= \sin x (\cos x-\gamma)+\ve (\cos x+\alpha)\,\dot x,
\]
performed in~\cite{L1}, with $\alpha$ and $\gamma$ real parameters and $\ve$ a small parameter.
Notice that in this problem the unperturbed Hamiltonian is not system~\eqref{E-HamSys}.
\\

To state our results we first fix some notation and definitions. A {\em Fourier polynomial of
degree $n$} is an element of the $2n$ dimensional real linear space generated by
$1,\sin(x),\ldots\sin(nx), \cos(x),\ldots,\cos(nx).$ It is well-known that this space is the same
as the space of degree $n$ two variable polynomials in $(\sin(x),\cos(x)).$ Given a Fourier
polynomial $P(x)=\sum_{i=0}^na_i\sin(ix)+b_i\cos(ix)$ we denote by $P^e(x)$ its even part, that
is, $P^e(x)=\sum_{i=0}^nb_i\cos(ix).$ Note that any even Fourier polynomial of degree $n$ can be
equivalently written as a degree $n$ polynomial in $\cos(x).$ At different points in the paper we
will choose the expression more suitable for our respective interest. From now on we denote by
$E(x)$ the integer part for any real number $x$.

Our first main result gives general upper bounds for the number of zeros of the first Melnikov
integral.


\begin{thmx}
\label{T-thm bounds}
 Consider the system
    \begin{equation}
      \label{E-penduSysGen}
           \sys{y,}{-\sin( x)+ \ve \sum_{\s=0}^{m}Q_{n,\s}(x) y^{\s},}
     \end{equation}
where  $Q_{n,\s}$ are Fourier polynomials of degree $n$ and let
$M^0:(0,2)\longrightarrow \R,$ and $M^{\pm}:(2,\infty)\longrightarrow
\R$ be their associated first Melnikov functions defined  by
\[
 M^0(h)  = \int_{\g_h} \sum_{\s=0}^{m}Q_{n,\s}(x) y^{\s}\, dx,\quad M^{\pm }(h)
  = \int_{\g_h^{\pm}} \sum_{\s=0}^{m}Q_{n,\s}(x)
 y^{\s}\, dx,
\]
 in $\mathcal {R}^0$ and $\mathcal {R}^\pm$, respectively. Then the
following statements hold:
\begin{enumerate}

\item[(a)]For all $h\in(2,\infty)$ \begin{align*}
    M^+(h)  &= \int_{\g_h^+} \sum_{\s=0}^{m}Q^e_{n,\s}(x) y^{\s} \, dx.
\end{align*}
 Moreover, if $M^+(h)$ is not identically zero in $(2,\infty)$ then it has at most
 $2n + 2m + E\left ({m}/{2}\right)+2$ zeros  counting
multiplicity in the interval $(2,\infty).$ The same result holds
    for $M^-(h).$

    \item[(b)] For all $h\in(0,2)$,
        \begin{align*}
            M^0(h) &=  \int_{\g_h} \sum_{\s=0}^{E((m-1)/{2})}Q^e_{n,2\s+1}(x)
            y^{2\s+1} \, dx.
        \end{align*}
        Moreover, if $M^0(h)$ is not identically zero in $(0,2)$ then it has at
most $2n +2E({(m-1)}/{2}) + 1$ zeros counting multiplicity in  $(0,2)$
\end{enumerate}
\end{thmx}

\smallbreak


The bounds given in Theorem~\ref{T-thm bounds} are not optimal. In the following two
theorems, we give optimal bounds for some particular cases in
the oscillatory region (Theorem \ref{T-thm inside}) as well as the
rotatory region (Theorem \ref{T-thm outside}). To this end, given
two natural numbers $s_1\le s_2$, we denote
$$o(s_1,s_2)=\left\{
    \begin{array}{ll}
        E\left({(s_2-s_1)}/{2}\right)-1, & \hbox{if $s_1$ and $s_2$ are even;} \\
        E\left({(s_2-s_1)}/{2}\right), & \hbox{otherwise.}
    \end{array}
                    \right.$$
A simple computation
 shows that $o(s_1,s_2)$ is the number of odd integers in $[s_1,
s_2]$ minus one.  Notice that $o(0,m)=E((m-1)/2).$ In case that $o(s_1,s_2)\ge 0$ we define
$l(s_1,s_2)$ as the first odd integer in $[s_1,s_2]$.


\begin{thmx}\label{T-thm inside} Consider the system
    \begin{equation*}
      \label{E-penduSysGen2}
           \sys{y,}{-\sin( x)+ \ve \sum_{s=s_1}^{s_2}Q_{n,s}(x) y^{s},}
     \end{equation*}
where  $Q_{n,s}$ are Fourier polynomials of degree $n$ and let $M^0:(0,2)\longrightarrow \R$ be
its associated first Melnikov function on the period annulus $\mathcal {R}^0$, defined by
\begin{equation*}
    M^0(h)  = \int_{\g_h} \sum_{s=s_1}^{s_2}Q_{n,s}(x) y^{s}\, dx.
\end{equation*}
Set $r=o(s_1,s_2)$ and $\ell= l(s_1,s_2)$ when $r\ge 0.$ Then,
\begin{enumerate}[(a)]

\item If $r=-1$ then the system has a center at the origin for all $\ve,$
and no limit cycles bifurcates from $\mathcal {R}^0$.

\item If $r\ge0$ then
$$M^0(h)=\int_{\gamma_h}\sum_{s=0}^{r}Q^e_{n,\ell+2s}(x)
y^{\ell+2s}\, dx$$ and it holds that:
\begin{enumerate}

\item [$(b_1)$]  If $0\le r<(\ell+3)/2$ and $M^0$ is not identically zero then
it has at most $n+2r$ zeros counting multiplicity. Moreover, if $r\le 2$ and $n>0$ then there
exist even Fourier polynomials $Q_{n,s}(x)$  such that the Melnikov function has exactly $n+2r$
zeros counting multiplicity.

\item [$(b_2)$] If $s_1=s_2$ is odd then there are at most $n$ limit
cycles that bifurcate from the period annulus. This bound is optimal.

\end{enumerate}
\end{enumerate}
\end{thmx}


Item $(b_1)$ in Theorem~\ref{T-thm inside} gives upper bounds for the number of zeros of the first
Melnikov function in $\mathcal {R}^0$. Statement $(b_1)$ also says that these bounds are optimal
when $n>0$ and $r\le 2$. In fact, we think that they are optimal for all $r$ when $n>0$, but we
have not been successful in proving it. In the case $n=0$ these bounds are not optimal because of
the following result:

\begin{proposition}\label{p-nova}
The system
    \begin{equation*}
      \label{E-penduSysGen3}
           \sys{y,}{-\sin( x)+ \ve \sum_{s=0}^{r}a_s y^{2s+1},\quad r\le30,}
     \end{equation*}
has at most $r$ limit cycles  bifurcating from the period annulus $\mathcal {R}^0$.
\end{proposition}

We suspect that the proposition holds for all $r$, but our proof relies on huge explicit
computations showing  that the family
$$\Big\{\int_{\gamma_h}y^{2s+1}\, dx\Big\}_{s=0}^{r}$$ is an extended
complete Chebyshev system in $(0,2)$. We have performed them  only until $r=30.$\\

The next theorem gives bounds for the number of limit cycles bifurcating in rotatory region~$\mathcal{R}^\pm.$


\begin{thmx}
\label{T-thm outside}
Consider the system
    \begin{equation*}
      \label{E-penduSysGen4}
           \sys{y,}{-\sin( x)+ \ve \big(Q_{n}(x) y^{2p+1}+\sum_{s=0}^{r}Q_{n,s}(x)y^{2s}\big),}
     \end{equation*}
where $p,r\in\N$ and $Q_{n,s}$ and $Q_{n}$ are Fourier polynomials of degree $n$ and let
$M^{\pm}:(2,\infty)\longrightarrow \R$ be its associated first Melnikov functions on the rotary
regions $\mathcal{R}^\pm.$  Assume also that $M^{\pm}(h)$ is not identically zero.  Then it has at
most $n+r+1$ zeros in $(2,\infty)$, counting  multiplicity.  This bound is  optimal on each of the
regions $\mathcal{R}^-$ and $\mathcal{R}^+.$
 Moreover, this upper bound can be reduced to $r$ when
 $Q_{n}(x)\equiv0$ and to
 $n$   when $Q_{n,0}(x)\equiv
 Q_{n,1}(x)\equiv\cdots Q_{n,r}(x)\equiv0$.  These upper bounds
are also sharp on each of the regions
$\mathcal{R}^-$ and $\mathcal{R}^+.$
\end{thmx}

 When finishing this paper we became aware of the book \cite{Morozov} where similar questions are
treated in detail.  To compare these results with ours we apply the above theorems to the simple
example
\begin{equation}\label{E-example}
           \sys{y,}{-\sin( x)+ \ve \sum_{s=0}^{3}\big(a_s+b_s\sin(x)+c_s\cos (x)\big)y^s,}
     \end{equation}
where $a_s,b_s,c_s\in\R.$ In the notation of Theorem~\ref{T-thm bounds}, $n=1$ and $m=3.$ Using item (b) of Theorem~\ref{T-thm bounds}  we get that in the oscillatory region the maximum number of zeroes, counting multiplicity, of a
nonvanishing Melnikov function $M^0$ is $2n+ 2  E(( m  -1)/2)+1=5.$  Item $(b_1)$ of Theorem~\ref{T-thm inside} improves this upper bound.  Indeed, in the notation of Theorem~\ref{T-thm inside}, $n=1$, $s_1=0$ and $s_2=3$. Then
$r=o(0,3)=1, \ell=l(0,3)=1$ and the hypothesis $0\le r <  (\ell+3)/2=2$ holds. Thus, the maximum
number of zeros, counting  multiplicity, is $n+2r=3$. Moreover, since $r\le2$
and $n>0$, this upper bound is sharp.

Contrary to our findings, Theorem 4.10 of \cite[p. 135]{Morozov} asserts that for the general
system \eqref{E-penduSysGen} the number of zeros of $M^0$ in the oscillatory region
$\mathcal{R}^0$ is at most $n+E((m-1)/2)$. Applying his result to system~\eqref{E-example} gives
an upper bound of~$2.$ Therefore, our results show that Theorem 4.10 is not correct. We want to
point out that this can be seen directly  without using Theorem~\ref{T-thm inside} by choosing some parameters for
which the corresponding Melnikov function $M^0$ has at  least 3 zeros. In this
situation we have that in $\mathcal{R}^0$,
\[
M^0(h)=a_1  \int_{\g_h} y \, dx+ c_1  \int_{\g_h} \cos(x) y \, dx+ a_3 \int_{\g_h}  y^3 \, dx+ c_3
\int_{\g_h} \cos(x) y^3 \, dx,
\]
and there exist values $a_1,c_1,a_3$ and $c_3$ such that $M^0$ has at least 3  simple zeros in
$(0,2)$ since these four Abelian integrals are linearly independent.

\smallbreak

In fact, the line of arguments in the beginning of our proof of Theorem~\ref{T-thm bounds} is similar to the one of the proofs in
\cite{Morozov1989,Morozov}. The Abelian integral $I(h)$ associated to \eqref{E-penduSysGen} can be expressed, in the rotary and the oscillatory regions, in terms of polynomials and the complete elliptic functions of
first and second kind,
\begin{align}
\label{E-Elliptic}
    K(k) =  \int_0^{\pi/2}\frac{d\theta}{\sqrt{1-k^2\sin^2(\theta)}}, \quad
     E(k) =  \int_0^{\pi/2} \sqrt{1-k^2\sin^2(\theta)}\,d\theta,
\end{align}
satisfying certain recurrence relations.  We then use this result together
with an upper bound on the number of zeros of functions of the form $P(k)E(k)+Q(k)K(k)$ in $(-1,1)$, where $P$ and $Q$ are polynomials given  in~\cite{Gasull2002}.

In contrast, Morozov studied these functions directly by complexifying the variable $k$, and
applying the argument principle to a suitable domain. This method is, indeed,  the one used to prove the results in \cite{Gasull2002}.  So the inaccuracy of the  upper bounds given in \cite{Morozov} appears to originate from some of the steps in the analysis of these complexified functions.

The proofs of Theorems~\ref{T-thm inside} and~\ref{T-thm outside} are based on criteria developed
in~\cite{Grau2011} and~\cite{Gasull2012b}, respectively. Both proofs show that certain subfamilies
of Abelian integrals associated  to \eqref{E-penduSysGen}  form a Chebyshev system. 

\bigskip

The  paper is organized as follows:  In Section \ref{pa} we prove Theorem~\ref{T-thm bounds}, in
Section \ref{pb} we give the proof of Theorem \ref{T-thm inside}, while Section \ref{pc} addresses the proof of Theorem \ref{T-thm outside}.  Section \ref{SS-simultaneous} is devoted to
simultaneous bifurcation of limit cycles. Notice that our main results give bounds for the number
of zeros of the corresponding Abelian integrals on each of the  regions $\mathcal{R}^0$ and
$\mathcal{R}^\pm$ by studying  them separately. We end the paper with some comments and results
showing the difficulties of studying the coexistence of limit cycles in these three regions.
This problem has also been addressed briefly in~\cite{Morozov1989,Sanders1986} for
some particular cases of system~\eqref{E-penduSysGen}.


\section{Proof of Theorem A}\label{pa}
We begin by studying Abelian integrals of the type

\begin{equation}
\label{E-integrals1}
    \int_{\g_h} \cos^n(x) y^r \, dx,
\end{equation}
where $r,n\in\N$ and $\g_h\subset\{{y^2}/{2}+1-\cos(x)=h\}$. We denote
\begin{align}
\label{E-integrals2}
    I_{n,r} (h) & =   \int_0^{\al(h)} \cos^n(x) \left(\sqrt{h-1+\cos(x)}\right)^r \, dx,
\end{align}
where $\h\in(0,\infty)$ and the integration boundary is given by
\begin{equation*}
    \a(h)=  \left\{
                  \begin{array}{l l}
                        \arccos(1-h)  &\text{ for } \h\in(0,2),\\
                        \pi                  &  \text{ for } \h\in(2,\infty).
                  \end{array}\right.
\end{equation*}
Furthermore, we denote by $I^0_{n,r}$  and $I^+_{n,r}$ the restrictions of $I_{n,r}$ to the
intervals $(0,2)$ and $(2,\infty)$, respectively. Moreover, we denote $I^-_{n,r}=(-1)^r I^+_{n,r}$.
These
integrals coincide, except for a multiplicative constant with the corresponding Abelian integral
\eqref{E-integrals1}.

As we will see, the integrals $I_{n,r}(h)$ can be written in terms of the complete elliptic
integrals of first and second kind $K$ and $E$, see~\eqref{E-Elliptic}. Our computations to prove
this fact are inspired by \cite{Morozov1989} and use several well-known properties of
these elliptic functions, see \cite{Byrd,Gradshteyn1981}.

 To
prove this property for $I_{n,r}(h)$, it is essential that $K$
and $E$ are closed under derivation, that is,  expressions of the
form
\begin{equation}
\label{E-Elliptic Poly}
    f(k) K(k) + g(k) E(k),
\end{equation}
where $f$ and $g$ are rational functions in $k$, remain of this form
after differentiation with respect to $k$. This is due to the fact
that the elliptic functions $K$ and $E$ satisfy the
Picard-Fuchs equations
\begin{equation}
\label{E-Elliptic Diff}
    \frac{d K}{d k} = \frac{E-(1-k^2)K}{k(1-k^2)},\qquad
    \frac{d E}{d k} = \frac{E- K}{k},
\end{equation}
see \cite{Byrd},     formulas 710.00 and 710.02.
Once we are able to express the integrals \eqref{E-integrals2} in
terms of $E$ and $K$, we may use a result derived in
\cite{Gasull2002} which provides an upper bound on the number of
zeros of expressions of the form \eqref{E-Elliptic Poly}.

\begin{theorem}[Theorem 1 in \cite{Gasull2002}]
\label{T-thm Gasull Elliptic} Let $f$ and $g$ be real polynomials of
degree at most $n$ and $m$, respectively, and let $k\in(-1,1)$. An
upper bound for the number of zeros of the function $f(k) K(k) +
g(k) E(k)$, taking into account their multiplicity, is $n+m+2$.
\end{theorem}

The next Lemma shows that for $n=0,1$ the integrals $ I^0_{n,1}$ and $ I^+_{n,1}$ can be expressed
as combinations of $E$ and $K$ with polynomial coefficients.


\begin{lemma}
\label{L-Integrals 0 1} Denote $L^0_n(h)=I^0_{n,1}(h)$ and $L^+_n(h)=I^+_{n,1}(h)$,
see~\eqref{E-integrals2}. Then the following statements hold:
\begin{enumerate}[(A)]
\item
Let $h\in(0,2)$, then
\begin{align*}
        L^0 _0(h) & = \sqrt{2}\left((h-2)K\left(\sqrt{{h}/{2}}\right) +
        2 E\left(\sqrt{{h}/{2}}\right)\right)
\end{align*}
and
\begin{align*}
        L^0 _1(h) & = \frac{\sqrt{2}}{3}\left((2-h)K\left(\sqrt{{h}/{2}}\right)
        + 2(h-1) E\left(\sqrt{{h}/{2}}\right)\right).
\end{align*}

\item Let $h\geq 2$, then
\begin{align*}
        L^+_0(h) 
                      & = 2\sqrt{h} E\left(\sqrt{{2}/{h}}\right)
\end{align*}
and \begin{align*}
        L^+_1(h) &  = \frac{2}{3} \sqrt{h}
                        \left((2-h)K\left(\sqrt{{2}/{h}}\right) + (h-1)
                        E\left(\sqrt{{2}/{h}}\right)\right).
\end{align*}
\end{enumerate}
\end{lemma}


\begin{proof}
(A)  The classical change of variables
\begin{equation*}
\label{E-cov}
    \xi = \arcsin\left(\sqrt{\frac{1-\cos(x)}{h}}\right),
\end{equation*}
see~\cite{Byrd}, allows us to rewrite the first integral $L^0 _0$ as
\begin{align*}
  L^0 _0&(h) =\int_0^{\arccos(1-h)}  \sqrt{h-1+\cos(x)} \, dx\\
                  &=\int_0^{\pi/2} \sqrt{h-h \sin^2(\xi)} \frac{\sqrt{2h}
                                             \cos(\xi)}{\sqrt{1-{h}\sin^2(\xi)/{2}}}d\xi
= \sqrt{2}h \int_0^{\pi/2} \frac{1-\sin^2(\xi)d\xi}{\sqrt{1-{h} \sin^2(\xi)/{2}}}
\end{align*}
Notice that
\begin{equation*}
\label{E-E'}
    E'\left(\sqrt{{h}/{2}} \right)\sqrt{{2}/{h}} = - \int_0^{\pi/2}
    \frac{\sin^2(\xi)d\xi}{\sqrt{1-{h} \sin^2(\xi)/{2}}},
\end{equation*}
and hence
\begin{equation*}
    K\left(\sqrt{{h}/{2}}\right) +  E'\left(\sqrt{{h}/{2}}\right)\sqrt{{2}/{h}}
        =\frac{L^0 _0(h)}{\sqrt{2}h}.
\end{equation*}
Moreover, in view of the fact that $E$ and $K$ satisfy the
differential equations \eqref{E-Elliptic Diff} we know that
    $E'\left(\sqrt{{h}/{2}}\right) =
    \frac{E(\sqrt{{h}/{2}})-K(\sqrt{{h}/{2}})}{\sqrt{{h}/{2}}}.$
Therefore,
\begin{align*}
    L^0 _0(h) &= \sqrt{2}h \left ( K(\sqrt{{h}/{2}}) + \sqrt{{2}/{h}}\, E'(\sqrt{{h}/{2}}) \right
    )\\
    &= \sqrt{2}h \left ( K(\sqrt{{h}/{2}}) +{2} \big(E(\sqrt{{h}/{2}})
                     -K(\sqrt{{h}/{2}})\big)/{h} \right )  \\
                   &= \sqrt{2} \big( (h-2)K(\sqrt{{h}/{2}}) + 2E(\sqrt{{h}/{2}}) \big),
\end{align*}
which proves the first assertion in (A). The first statement in (B)
is a straightforward calculation. Indeed,
\begin{align*}
    L^+_0(h) &=  \sqrt{h}\int_0^{\pi} \sqrt{1-(1-\cos(x))/h } \,dx=
    \sqrt{h}\int_0^{\pi} \sqrt{1-{2}\sin^2\left({x}/{2}\right)/{h}}\, dx\\
                    &= 2\sqrt{h}\int_0^{\pi/2} \sqrt{1-{2}\sin^2(\theta)/{h}}\,
                     d\theta = 2\sqrt{h} E\left(\sqrt{{2}/{h}}\right).
\end{align*}
To show the second statements in (A) and (B), we make a general
observation which is true in both cases, that is, in the oscillatory
as well as the rotary region. We drop the superscripts for lighter
notation, and notice that in view of
\begin{equation*}
    (h-1) L_0(h) + L_1(h)= \int (h-1+\cos(x))^{3/2}\, dx
\end{equation*}
we find that $((h-1) L_0(h) + L_1(h))' = \frac{3}{2} L_0(h)$ and
therefore
\begin{equation}
\label{E-I1 in terms of L0'}
    L_1'(h)=\frac{1}{2} L_0(h) - (h-1) L_0'(h).
\end{equation}
To prove the second statement of (A) we proceed by making an Ansatz of the form $S^0 (h)=a(h)
K\left(\sqrt{{h}/{2}}\right) + b(h) E\left(\sqrt{{h}/{2}}\right)$, where $a(h)$ and $b(h)$ are
real polynomials in $h$. Differentiating this expression and equating it with the right-hand side
of \eqref{E-I1 in terms of L0'}, we obtain a linear system of differential equations in $a(h)$ and
$b(h)$. Comparing coefficients of $K$ and $E$ we obtain the solution $a(h) =\frac{\sqrt{2}}{3}
(2-h)$ and $b(h) =\frac{\sqrt{2}}{3}2  (h-1)$. To make sure that the corresponding solution $S^0
(h)= \frac{\sqrt{2}}{3} \left((2-h)K\left(\sqrt{{h}/{2}}\right) + 2
(h-1)E\left(\sqrt{{h}/{2}}\right)\right)$ is the correct one, it suffices to show that $\lim_{h\to
2^-} S^0 (h)- L^0 _1(h)=0.$ This holds because
$$\begin{array}{ll}
\lim_{h\to 2^-}\left|(2-h)K\left(\sqrt{{h}/{2}}\right)\right|
        &=\lim_{h\to 2^-}\sqrt{(2-h)}\int_0^{\pi/2}\sqrt{2}\sqrt{\frac{2-h}{2-\sin^2\theta}}\,d\theta\\
        &\le \lim_{h\to 2^-}\sqrt {2}\sqrt{2-h}\frac{\pi}{2}=0
\end{array}$$
and therefore $\lim_{h\to 2^-}S^0 (h)= 2\frac{\sqrt {2}}{3}E(1) =2\frac{\sqrt {2}}{3}.$ Moreover,
a simple computation shows that
$$
   \lim_{h\to 2} L^0 _1(h)  = \int_0^{\pi} \sqrt{1+ \cos(x)}\cos(x)\, dx\\
                       = 2\frac{\sqrt {2}}{3}.
$$
To prove the second statement in (B)  we proceed as in case (A), and make an Ansatz of the form
$S^+(h)=\sqrt{h}\left(a(h) K\left(\sqrt{{2}/{h}}\right) + b(h)
E\left(\sqrt{{2}/{h}}\right)\right)$. We then solve the corresponding system of differential
equations and obtain the solution $a(h)=\frac{2}{3} (2-h)$ and $b(h)=\frac{2}{3} (h-1)$. As above,
simple computations show that $ \lim_ {h\to 2^+}S^+(h)- L^+_1(h) = 0.$ This ends the proof of the
Lemma.

\end{proof}


\begin{lemma}
\label{L-integral indu} For any positive real number $r> 0$ and any
$m\in\N$ it holds that
    \begin{equation}
    \label{E-int ind}
         \int t^m(t+s)^r dt=(t+s)^{r+1} V_m(t,s),
     \end{equation} where
    \begin{equation*}
        V_m(t,s)=\frac{1}{r+m+1} \Big( t^m - m s V_{m-1}(t,s)\Big),
    \end{equation*}
   is a homogeneous polynomial of degree $m$ with $V_0(t,s)={1}/{(r+1)}$ and
    $V_1(t,s)=\frac{1}{r+2}\left (t - s \frac{1}{r+1}\right)$.
\end{lemma}


\begin{proof}
Let us denote  $V_m=V_m(t,s)$ and
    \begin{equation*}
        U_m=U_m(t,s)=\int t^m(t+s)^r dt,
    \end{equation*}
for $m\in\N$. Integrating by parts and rearranging the terms we find
that
\begin{equation*}
    U_m =\frac{1}{r+m+1}\Big(t^m(t+s)^{r+1} - m s U_{m-1}\Big).
\end{equation*}
Now the claim follows by induction. Indeed, a direct calculation
shows  that $U_0=\frac{1}{r+1} (t+s)^{r+1}= (t+s)^{r+1} V_0$  and
similarly $U_1= (t+s)^{r+1} V_1$. Now assume that statement
\eqref{E-int ind} holds for all $i\leq m$. Then, integrating by
parts we find that
\begin{align*}
    U_{m+1} 
                &= t^{m+1}\frac{(t+s)^{r+1}}{r+1} - \frac{m+1}{r+1}\left(U_{m+1} + s U_m\right),
\end{align*}
which yields
\begin{align*}
    U_{m+1} &= \frac{1}{m+r+1}\left(t^{m+1} (t+s)^{r+1} - (m+1) s U_m\right)\\
                &= \frac{1}{m+r+1}\left(t^{m+1} (t+s)^{r+1} - (m+1) s (t+s)^{r+1} V_m\right)\\
                &= (t+s)^{r+1}\frac{1}{m+r+1}\left(t^{m+1}  - (m+1) s  V_m\right)\\
                &= (t+s)^{r+1}V_{m+1},
\end{align*}
where we have used the induction hypothesis in the second equality
of the above expression.
\end{proof}

Now we are ready to prove the desired expression for $I_{n,1}(h)$ with any
$n\in\N$.


\begin{proposition}
\label{P-Ln Elliptic}
    Denote $L^0_n(h)=I^0_{n,1}(h)$ and $L^+_n(h)=I^+_{n,1}(h)$, see~\eqref{E-integrals2}.
    Then there exist real polynomials  $P^0_n,P^+_n, Q^0_n$ and $Q^+_n$ of degree $n\in\N$
     such that
\begin{align*}
 L^0 _n(h)&=  P^0 _n(h) K\left(\sqrt{{h}/{2}}\right)
  + Q^0 _n(h) E\left(\sqrt{{h}/{2}}\right)\quad \mbox{when} \quad h\in(0,2),\\
 L^+_n(h)&= \sqrt{h} \left(P^+_n(h) K\left(\sqrt{{2}/{h}}\right)
+ Q^+_n(h) E\left(\sqrt{{2}/{h}}\,\right)\right) \quad \mbox{when} \quad h\in(2,\infty).
\end{align*}
\end{proposition}


\begin{proof}
Lemma \ref{L-Integrals 0 1} proves the result  for $n=0,1.$ Now we
claim that for $n>1$ we have
\begin{equation*}
    L_n(h) =a_{1}(h) L_{n-1}(h) +  a_{2}(h) L_{n-2}(h) + \dots +  a_{n-1}(h)L_1(h),
\end{equation*}
where $a_i(h)$ are polynomials with degree $i$. Note that
\begin{align*}
    L_n(h) &=\int_0^{\al(h)} \cos^n(x) \sqrt{h-1+\cos(x)} \, dx \\
               &=L_{n-2}(h) -  \int_0^{\al(h)} \cos^{n-2}(x) \sqrt{h-1+\cos(x)} \sin^2(x)\, dx.
\end{align*}
We want to perform integration by parts in the second integral. To
this end, we use Lemma \ref{L-integral indu} with $r=1/2$ and
$s=h-1$ to obtain
\begin{equation*}
    \int \cos^{m}(x) \sqrt{h-1+\cos(x)} \sin(x)\, dx = \sqrt{h-1+\cos(x)}\, P_{m+1}(\cos(x),h-1),
\end{equation*}
where $P_{m+1}$ is homogeneous of degree $m+1$ with $P_{m+1}(z,0) =
\frac{2}{3+2m} z^{m+1}.$ Therefore, we have that
\begin{align*}
    L_n(h) &=L_{n-2}(h) + \int_0^{\al(h)} \sqrt{h-1+\cos(x)}\, P_{n-1}(\cos(x),h-1) \cos(x)\, dx\\
              &=L_{n-2}(h) + \int_0^{\al(h)}\sqrt{h-1+\cos(x)}
                                      \sum_{i=0}^{n-1} b_i \cos^i(x) (h-1)^{n-i-1} \cos(x)\, dx\\
              &=L_{n-2}(h) + \sum_{i=0}^{n-1} b_i (h-1)^{n-i-1}
                                      \int_0^{\al(h)}\sqrt{h-1+\cos(x)} \cos^{i+1}\, dx,\\
              &=L_{n-2}(h) + \sum_{i=0}^{n-1} b_i (h-1)^{n-i-1} L_{i+1}(h),
\end{align*}
where $b_{n-1} = \frac{2}{2n-1}.$ Therefore we get
$$L_n(h)=\frac{2n-1}{2n-3}\left(L_{n-2}(h) + \sum_{i=0}^{n-2} b_i (h-1)^{n-i-1}
L_{i+1}(h)\right)$$ and the claim is proved. Now the proposition
follows directly by induction.
\end{proof}


\begin{lemma}
\label{L-differentiating P K + Q E}
    Let $k^2={2}/{h}$ and for $h\in (2,\infty)$ consider
    \begin{equation*}
    \label{E-Lambda}
        \Lambda_m(h):= \sqrt{h}\big(P_m(h) K(k) + Q_m(h) E(k)\big),
    \end{equation*}
    where $P_m,Q_m$ are real polynomials of degree $m$.
    Then, the $n^{th}$-derivative of this expression is given by
    \begin{equation*}
    \label{E-Lambda Diff}
        \Lambda_m^{(n)}(h) = \frac{1}{(h(h-2))^n}\sqrt{h}\big(P_{m+n}(h) K(k) + Q_{m+n}(h) E(k)\big),
    \end{equation*}
    for all $n\in\N$,  where $P_{m+n}$ and $Q_{m+n}$ are real polynomials of degree $m+n$.
\end{lemma}

\begin{proof}

The equality is obviously true for $n=0.$ The result follows
directly by induction using  \eqref{E-Elliptic Diff}.

\end{proof}


\begin{proposition}
\label{P bound outside}
    Let $h\in(2,\infty),\,\, k^2={2}/{h}$
and $n,s,r\in \N.$ Then there exist polynomials  $Z_s,P_{n+r}$ and $Q_{n+r}$ of degrees $s$ and
$n+r$ such that $I^+_{n,2s}(h)= Z_s(h)$ and $I^+_{n,2r+1}(h)=\sqrt{h} \left(P_{n+r}(h) K(k)+
Q_{n+r}(h) E(k)\right).$ Moreover,  any nontrivial function of the form
\[
\tilde Z_s(h)+\sqrt{h} \left(\tilde P_{n+r}(h) K(k)+ \tilde Q_{n+r}(h) E(k)\right),
\]
where $\tilde Z_s,\tilde P_{n+r}$ and $\tilde Q_{n+r}$ are also polynomials with respective degrees
$s,n+r$ and $n+r$, has at most  $2(n+r)+3s+4$ zeros, counting multiplicity.
\end{proposition}


\begin{proof}
For lighter notation we drop the superscripts and observe that
\begin{align*}
    I_{n,2s}(h) &= \int_0^{\pi}\cos^n(x) \left(\sqrt{h-1+\cos(x)}\right)^{2s} \, dx \\
                    &= \sum_{i=0}^s c_i (h-1)^{s-i} \int_0^{\pi} \cos^{i+n}(x) \, dx=Z_s(h),
\end{align*}
where $Z_s$ is a real polynomial of degree $s$. Furthermore,
\begin{align*}
    I_{n,2r+1}(h) &= \int_0^{\pi}\cos^n(x) \left(\sqrt{h-1+\cos(x)}\right)^{2r+1} \, dx \\
                        &= \sum_{i=0}^r d_i (h-1)^{r-i} \int_0^{\pi}
                        \cos^{i+n}(x)\sqrt{h-1+\cos(x)} \, dx\\
                        &= \sqrt{h} \left(P_{n+r}(h) K\left(\sqrt{{2}/{h}}\right)+
                         Q_{n+r}(h) E\left(\sqrt{{2}/{h}}\right)\right),
\end{align*}
in view of Proposition \ref{P-Ln Elliptic}. Thus, the first statement
of the Proposition is proved. Differentiating the above expression $s+1$ times using
Lemma \ref{L-differentiating P K + Q E}, we obtain
$$ \begin{array}{ll}
I_{n,2r+1}^{(s+1)}(h)&=\frac{\sqrt{h}}{(h(h-2))^{s+1}}\big(P_{n+r+s+1}(h) K(k)
+ Q_{n+r+s+1}(h) E(k)\big)\\ &=
\frac{\sqrt{h}}{(h(h-2))^{s+1}}\big(P_{n+r+s+1}(\frac{2}{k^2}) K(k)
+ Q_{n+r+s+1}((\frac{2}{k^2})) E(k)\big).
\end{array}$$
Thus, any zero of $I_{n,2r+1}^{(s+1)}(h)$ corresponds to a positive zero of
$$ P_{n+r+s+1}\left({2}/{k^2}\right) K(k) + Q_{n+r+s+1}\left({2}/{k^2}\right)
E(k),$$
which is also a positive zero of
$$P_{2(n+r+s+1)}(k) K(k) + Q_{2(n+r+s+1)}(k)E(k),$$
for certain even polynomials  $P_{2(n+r+s+1)}$ and $Q_{2(n+r+s+1)}$ of degree $2(n+r+s+1).$
By Theorem \ref{T-thm Gasull Elliptic} we know that the number of
zeros of this last expression in $(-1,1)$ is bounded by $4(n+r+s+1)
+2.$ Since the expression is even, we obtain that the number of
zeros of $I_{n,2r+1}^{(s+1)}(h)$ in $(2,\infty)$ is bounded by $2(n+r+s+1) +1$,
and obtain the desired result applying Rolle's Theorem $s+1$ times.
\end{proof}


\begin{proposition}
\label{P bound inside}
 Let $h\in(0,2),\,\,k^2={h}/{2}$
 and consider integrals of the form
 \begin{equation*}
         I_{n,2r+1}^0 (h) = \int_0^{\arccos(1-h)} \cos^n(x)
         \left(\sqrt{h-1+\cos(x)}\right)^{2r+1} \, dx.
 \end{equation*}
 Then  $I_{n,2r+1}^0 (h)=P_{n+r}(h) K(k)+ Q_{n+r}(h) E(k)$ for
 certain polynomials $P_{n+r}$ and $Q_{n+r}$ of degree $n+r.$
 Moreover, the number of zeros of
 $I_{n,2r+1}^0 (h)$, counting multiplicity, is less than or equal to
$2(n+r)+1.$
\end{proposition}


\begin{proof}
For $r=0$,  we know that $I_{n,1}^0 (h)=L_n^0 (h)= P_{n}(\frac{2}{k^2}) K(k) +
Q_{n}(\frac{2}{k^2}) E(k)$, in view of Proposition \ref{P-Ln Elliptic}. For $r\geq 1$ we find that
\begin{align*}
    &I_{n,2r+1}^0 (h) = \int_0^{\arccos(1-h)} \cos^n(x)\left(\sqrt{h-1+\cos(x)}\right)^{2r+1}  \, dx\\
                        &= \sum_{i=0}^r a_i (h-1)^{r-i}
                                 \int_0^{\arccos(1-h)} \cos^{n+i}(x)\sqrt{h-1+\cos(x)} \, dx\\
                        &= \sum_{i=0}^r a_i (h-1)^{r-i}  I_{n+i,1}^0 (h)= \sum_{i=0}^{r}a_i
                        (h-1)^{r-i}\left (P_{n+i}(h) K(k) + Q_{n+i}(h) E(k) \right)\\
                    &=P_{n+r}(h) K(k)+ Q_{n+r}(h) E(k)=P_{n+r}(2k^2) K(k)+ Q_{n+r}(2k^2) E(k)\\
                    &=P_{2(n+r)}(k) K(k)+ Q_{2(n+r)}(k) E(k),
\end{align*}
for certain even polynomials $P_{2(n+r)}$ and $Q_{2(n+r)}$ of
degree $2(n+r).$ In view of Theorem \ref{T-thm Gasull Elliptic} we
conclude that $P_{2(n+r)}(k) K(k)+ Q_{2(n+r)}(k) E(k)$ has at most
$4(n+r)+2$ zeros in $(-1,1)$. Since this expression is even we conclude that it
has at most $2(n+r)+1$ positive zeros and the result follows.
\end{proof}


\begin{proof}[Proof of Theorem \ref{T-thm bounds}]

We prove item (a) for $M^+.$ The proof for $M^-$ follows in the same
way. To prove the first statement  it suffices to show that for any
$j,i\in\N$ we have
$$\int_{\gamma_h}\sin (jx)y^{i}\, dx\equiv 0$$ on $(2,\infty).$ This
holds because $\sin (jx)\left(\sqrt{h-1+\cos(x)}\right)^{i}$ is an odd function
and
$$\int_{\gamma_h}\sin (jx)y^{i}\, dx=
    \int_{-\pi}^{\pi}\sin(jx)\left (\sqrt{h-1+\cos(x)}\right)^{i} \, dx =0.
$$
Now applying Proposition \ref{P bound outside} we obtain that
$$M^+(h)=Z_s(h)+\sqrt{h} \left(P_{n+r}(h) K(k)+ Q_{n+r}(h)
E(k)\right)$$ for some polynomials $Z_s, P_{n+r}$ and $Q_{n+r}$ of degree $s$ and $n+r.$ Here $s$
and $r$ are the largest natural numbers such that $2s\le m$ and $2r+1\le m$, respectively. That
is, $s=E(\frac{m}{2})$ and $r=E(\frac{m-1}{2}).$ From Proposition \ref{P bound outside} we obtain
that the number of zeros of $M^+(h)$ in $(2,\infty)$ is bounded by
$$2n+2E\left(\frac{m-1}{2}\right)
+3E\left(\frac{m}{2}\right)+4=2n+2m+E\left(\frac{m}{2}\right)+2.$$ To prove  the first statement
of item $(b)$ we note that for any $h\in (0,2)$,  for any $i\in \N$ and for any smooth function
$f$ we have that $\int_{\gamma_h}f(x)y^{2i}\, dx =0.$ This is a direct consequence of the symmetry
with respect the $x$-axis of the orbit $\gamma_h$ and Green's Theorem. Indeed, we have
$$\int_{\gamma_h}f(x)y^{2i}\, dx=\iint_{Int(\gamma_h)}2if(x)y^{2i-1}\, dx\, dy=0.$$
To finish the proof of the first statement we need to show that for any $j,i\in\N$ we have
$$\int_{\gamma_h}\sin (jx)y^{2i+1}\, dx\equiv 0$$ on $(0,2).$ Again  this a consequence of Green's
Theorem and the symmetry (this time with respect to the $y$-axis) of the orbit $\gamma_h$, since
$$\int_{\gamma_h}\sin(jx)y^{2i+1}\, dx=\iint_{Int(\gamma_h)}(2i+1)\sin(jx)y^{2i}\, dx\, dy=0.$$
Lastly, setting $k=\sqrt{{h}/{2}}$ we obtain  from Proposition \ref{P bound inside} that
$$M^0(h)=P_{n+r}(h) K(k)+ Q_{n+r}(h) E(k)$$
for  certain polynomials $P_{n+r}$ and $Q_{n+r}$ of degree $n+r.$ Now $r$ is the largest integer
satisfying $2r+1\le m$, that is, $r=E(\frac{m-1}{2}).$ Then, using again Proposition~\ref{P bound
inside} we obtain that the number of zeros of $M^0(h)$ in $(0,2)$ is bounded by
$2\left(n+E(\frac{m-1}{2})\right) +1.$ This ends the proof of Theorem \ref{T-thm bounds}.
\end{proof}


\section{Proof of Theorem B}\label{pb}

We start with some  definitions and known results.

\begin{definition}

Let $f_0,f_1, \dots f_{n-1}$ be analytic functions on an open interval~$L$.
\begin{enumerate}[(a)]
    \item  $(f_0,f_1, \dots f_{n-1})$ is a \emph{Chebyshev system} (T-system) on $L$ if
    any nontrivial linear combination
    \begin{equation*}
            \alpha_0 f_0(x)+ \dots + \alpha_{n-1} f_{n-1}(x)
    \end{equation*}
             has at most $n-1$ isolated zeros on $L$.
    \item  $(f_0,f_1, \dots f_{n-1})$ is a \emph{complete Chebyshev system} (CT-system) on $L$ if
               $(f_0,f_1, \dots f_{k-1})$ is a T-system for all $k=1,2,\dots n$.
    \item  $(f_0,f_1, \dots f_{n-1})$ is an \emph{ extended complete Chebyshev system} (ECT-system)
    on $L$ if, for all $k=1,2,\dots n$, any nontrivial linear combination
    \begin{equation*}
            \alpha_0 f_0(x)+ \dots + \alpha_{k-1} f_{k-1}(x)
    \end{equation*}
             has at most $k-1$ isolated zeros on $L$ counting multiplicity.
\end{enumerate}

 \end{definition}

It is clear that if $(f_0,f_1, \dots f_{n-1})$ is an ECT-system on $L$, then it is also a
CT-system on $L$.
However, the reverse implication is not true in general.
In order to show that a set of functions is a T-system, the notion of the \emph{Wronskian}
proves to be extremely useful.

\begin{definition}
Let $f_0,f_1, \dots f_{k-1}$ be analytic functions on an open interval $L$.
The \emph{continuous Wronskian} of $(f_0,f_1, \dots f_{k-1})$ at $x\in L$ is
\begin{equation*}
    W[f_0,\dots f_{k-1}](x)=det\bigg(f_j^{(i)}(x)\bigg )_{0\leq i,j\leq k-1}
    \end{equation*}
The \emph{discrete Wronskian} of $(f_0,f_1, \dots f_{k-1})$ at $(x_0,\dots, x_{k-1})\in L^k$ is
\begin{equation*}
    D[f_0,\dots f_{k-1}](x_0,\dots,x_{k-1})=det\bigg(f_j(x_i)\bigg)_{0\leq i,j\leq k-1}
\end{equation*}
\end{definition}

For the sake of brevity we use the shorthand $x_0,x_1,\dots, x_{k-1} = \mb{x_k}$. Recall that
if the functions $f_i$
are linearly dependent, so are the columns of $W$ and therefore $W[\mb{f_k}]=0$. The reverse
implication is not true in general.
However, if the $f_i$ are analytic then the vanishing of $W$ implies linear dependence (This
is due to Peano, and there are other,
 more sophisticated criteria due Bocher, Wollson,  etc.). The next result is well-known,
 cf.~\cite{KarlinStudden}.

\begin{lemma}
The following equivalences hold:
\begin{enumerate}[(a)]
    \item $(f_0,f_1, \dots f_{n-1})$ is a CT-system on $L$ if and only if for all  $k=1,2,\dots n$
    \begin{equation*}
        D[\mb{f_k}](\mb{x_k}) \neq 0
            \text{ for all } \mb{x_k} \in L^k \text{ such that } x_i\neq x_j \text{ for } i\neq j.
    \end{equation*}
    \item $(f_0,f_1, \dots f_{n-1})$ is an ECT-system on $L$ if and only if for all  $k=1,2,\dots n$
    \begin{equation*}
        W[\mb{f_k}](x) \neq 0
            \text{ for all } x \in L.
    \end{equation*}
\end{enumerate}

\end{lemma}

To study the limit cycles of  equation  \eqref{E-pendu} in the
oscillatory region $\mathcal{R}^0$ we  will repeatedly use a result
introduced by Grau, Ma\~{n}osas and Villadelprat in \cite{Grau2011},
which we state in Theorem \ref{T-THM} below. It allows one to deduce
Chebyshev properties for certain Abelian integrals from the
Chebyshev properties of the corresponding integrands. We state here
this result for the particular case of potential even systems. To
fix notation, consider $V$ an analytic even function defined in a
neighborhood of the origin that has a local non degenerate minimum
at $0$, and  assume  that $V(0)=0.$ That is, $V$ satisfies
$V(x)=V(-x),\,\,V(0)=V'(0)=0$ and $V''(0)>0.$ Consider the
associated Hamiltonian system given by $H(x,y)={y^2}/{2}+V(x)$. Then
the origin of $\R^2$ is a critical point of center type and there
exists a punctured neighborhood $\P$, the so-called {\it period
annulus}, of the origin which is foliated by ovals $ \gamma_h
\subset \{H(x,y)=h\}$. Thus, the set of ovals inside the period
annulus can be parametrized by the energy levels $h\in (0,h_0)$ for
some $h_0\in (0,\infty]$ and the projection on the $x$-axis of the
period annulus is a symmetric interval $(-x_r,x_r)$ with
$V(x_r)=h_0.$  The following result plays a key role
 in our analysis.

\begin{theorem}[\THM]
\label{T-THM}
Let us consider the Abelian integrals
\begin{equation*}
    I_i(h)=\int_{\gamma_h}f_i(x)y^{2v-1}\, dx, \quad i=0,1,\dots, n-1,\; v\in\N,
\end{equation*}
where for each $h\in(0,h_0)$, $\gamma_h$ is the oval surrounding the
origin inside the level curve $\{{y^{2}}/{2}+V(x)=h\}$ where $V$ is
an analytic even function with $V(0)=V'(0)=0,\,\,V'(x)> 0$ for all
$x\in (0,x_r),$ and  $V''(0)>0.$ Define
\begin{equation*}
    \ell_i(x)=\frac12\left(\frac{f_i(x)}{V'(x)}-\frac{f_i(-x)}{V'(-x)}\right).
\end{equation*}
Then $(I_0,\dots, I_{n-1})$ is an ECT-system on $(0,h_0)$ if $(\ell_0,\dots, \ell_{n-1})$ is a
CT-system on $(0,x_r)$ and $n<v+2$.
\end{theorem}

The authors of \cite{Grau2011} point out that if the
condition $s> n-2 $ does not hold, there is a
procedure to obtain a new expression for the same set of Abelian
integrals for which the corresponding $s$ is large enough to verify
the inequality. We review this result here (stated for potential
systems)  for the convenience of the reader.

\begin{lemma}[\LEM]
\label{L-LEM} Let $\g_h$ be an oval inside the level curve $\{y^2/2+V(x)=h\}$, and consider a
function $F$ such that $F/V'$ is analytic at $x=0$. Then, for any $k\in\N$,
    \begin{equation*}
    \label{e-subibaja}
        \int_{\g_h} F(x) y^{k-2} \, dx = \int_{\g_h} G(x) y^k \, dx,
    \end{equation*}
where $G(x) = \frac{1}{k} \left(\frac{F(x)}{V'(x)}\right)'(x)$.
\end{lemma}


We now apply this theory to Hamiltonian systems \eqref{E-HamSys} corresponding to pendulum-like
 equations of type \eqref{E-pendu}, with total energy given by \eqref{E-Hamiltonian}. The
 following is a useful auxiliary result.


\begin{lemma}
\label{L-Tche} For any $j\in \N$, consider the functions $I_j:(0,2)\longrightarrow\R$ defined by
$I_j(h)=\int_{\gamma_h}\frac{\cos^j(x)}{\sin^{2p}(x)}y^{2v-1}\, dx$ where $p,v\in\N$ and
$\g_h=\{{y^2}/{2}+1-\cos(x)\}.$ If $q<v+1$  then the family $(I_0,I_1,\ldots I_q)$ is an
ECT-system on $(0,2).$ Consequently, if $q<v+1$ then for any polynomial $P$ of degree $q$ the
Abelian integral
\begin{equation}
\label{E-AbelianInt}
    I(h)=\int_{\gamma_h}\frac{P(\cos(x))}{\sin^{2p}(x)}y^{2v-1}\, dx
\end{equation}
has at most $q$ isolated zeros in $(0,2)$, counting
multiplicity,
and it is identically zero if and only if $P$ is identically zero.
\end{lemma}

\begin{proof} Consider the odd functions $\ell_j(x)=\frac{\cos^j(x)}{\sin^{2p+1}(x)},$
for $j=0,\dots q$. In view of
Theorem \ref{T-THM} it suffices to show that $(\ell_0,\dots,
\ell_{q})$ is a CT-system on $(0,\pi).$ However, since $\sin (x)$
has no zeros in $(0,\pi)$ this is equivalent to showing that
$1,\cos(x),\ldots,\cos^q(x)$ is an CT-system on $(0,\pi).$ But this
is a direct consequence of the facts that $(1,x,\ldots,x^q)$ is an
ECT-system on $(-1,1)$ and $\cos(x)$ is a diffeomorphism between
$(0,\pi)$ and $(-1,1).$
\end{proof}


Due to the particular structure of the Hamiltonian under
consideration, the function $G$ in   Lemma \ref{L-LEM}  has a rather
simple form which reveals an interesting structural property of the
Abelian integrals \eqref{E-AbelianInt}. To see this,  let $A$ be
the set of real analytic functions on $(0,\pi)$  and consider the
linear operator $\L: A \longrightarrow A$ defined as
\begin{equation}
 \label{E-operatorL}
     \L[f(x)] :=\left (\frac{f(x)}{\sin(x)}\right)'.
\end{equation}
From now on  $\L^j$ denotes the composition of the operator $j$
times and $\L^0=\id$.


\begin{lemma}
\label{L-operatorL} Consider the operator $\L$ defined in \eqref{E-operatorL}.  Then the following
statements hold:

\begin{equation}
    \label{E-Lj}
    \L^j[f \cos(x)] = \L^j[f]\cos(x) - j\L^{j-1}[f]\quad \text{ for all } j\in\N. \tag{A}
\end{equation}

\begin{equation}
    \label{E-Ljm}
         \L^j[\cos(mx)] = \dfrac{P_{j,m}(\cos(x))}{\sin^{2j}(x)} \quad \text{ for all }
          j\geq m, \tag{B}
    \end{equation}
 where $P_{j,m}$ is a polynomial of degree $j-m$ satisfying the relation
    \begin{align}
    \label{E-Pjm}
        P_{j+1,m}(u) & = -P'_{j,m}(u) (1-u^2)-(2j+1)uP_{j,m}(u)
    \end{align}
with $P_{m,m} = K(m) \in \R,$ where
\begin{equation}
\label{E-Km}
    K(m+1) = -(2m+1) K(m), \quad K(0)=1.
\end{equation}
\end{lemma}


\begin{proof}
The proof of statement \eqref{E-Lj} is a straightforward  induction in $j$ using the
fact that the
operator $\L$ is linear which we omit for the sake of brevity.

The  proof of statement \eqref{E-Ljm} follows by induction as well. We start
with the base case when $j=m$ and the claim that
\begin{equation}
\label{E-Lmcosm}
          \L^m[\cos(mx)] = \frac{K(m)}{\sin^{2m}(x)},
\end{equation}
where the real number $K(m)$ is defined as above. To prove the claim, let us start with some
preliminary considerations. Notice that, in view of the identity $$\cos(mx)=\cos((m-1)x)\cos(x)
- \sin((m-1)x)\sin(x)$$ and the fact that the operator $\L$ is linear, we find that
\begin{align*}
    \L[\cos(mx)] &= \L[\cos((m-1)x)]\cos(x)-\cos((m-1)x)-\left(\sin((m-1)x)\right)'\\
                        &= \L[\cos((m-1)x))]\cos(x)-m\cos((m-1)x),
\end{align*}
where we have used statement \eqref{E-Lj} in the first equality. In view of this relation and
using statement \eqref{E-Lj} once more, we obtain that
\begin{align}
\label{E-Lmcosm2}
    \L^m[\cos(mx)] &= \L^{m-1}[\L[\cos((m-1)x)]\cos(x)] - m \L^{m-1}[\cos((m-1)x)]\notag\\
                             &= \L^m[\cos((m-1)x)]\cos(x)-(2m-1)\L^{m-1}[\cos((m-1)x)].
\end{align}
We are now ready to prove the claim by induction in $m$. The base case for $m=0$ holds trivially
with $K(0)=1$. Assuming that the statement is true for all $m\in\N$, we prove the inductive step
using the identities \eqref{E-Lmcosm} and \eqref{E-Lmcosm2} derived above. Indeed,
\begin{align*}
    \L^{m+1}&[\cos((m+1)x)] = \L^{m+1}[\cos(mx))]\cos(x)-(2m+1)\L^m[\cos(mx)] \\
                                          &= \L\left[ \frac{K(m)}{\sin^{2m}(x)}\right] \cos(x)
                                                -(2m+1) \frac{K(m)}{\sin^{2m+2}(x)}\\
                                          &= \frac{-(2m+1)K(m)\cos^2(x)}{\sin^{2m+2}(x)}
                                               -\frac{(2m+1) K(m)(1-\cos^2(x))}{\sin^{2(m+1)}(x)}\\
                                          &= \frac{K(m+1)}{\sin^{2(m+1)}(x)} ,
\end{align*}
which proves the claim with  $K(m+1)=-(2m+1)K(m)$ as defined in \eqref{E-Km}. Let us proceed with the
proof of statement \eqref{E-Ljm}. Assuming that this statement holds for $j\in \N$, the inductive step
follows immediately from the definition of $\L$. Indeed,
\begin{align*}
    \L^{j+1}[\cos(mx)]&= \L\left[ \frac{P_{j,m}(\cos(x))}{\sin^{2j}(x)}\right] \\
                       &= \frac{-P'_{j,m}(\cos(x))(1-\cos^2(x))-(2j+1)\cos(x)P_{j,m}
                       (\cos(x))}{\sin^{2(j+1)}(x)},
\end{align*}
which in view of \eqref{E-Pjm} concludes the proof.
\end{proof}


\begin{proof}[Proof of  Theorem \ref{T-thm inside}]

First we prove statement $(a)$. If $r=-1$ then $s_1=s_2=2p$ for some
natural integer $p.$ Then the result follows directly from the fact
that in this situation the system (\ref {E-penduSysGen}) is
reversible with respect the $x$-axis and therefore it has a center
at the origin for all $\ve.$

$(b)$ Now $r\ge 0.$ The above argumentation shows that all Abelian integrals
$\int_{\gamma}Q_{n,s}y^s$ are identically zero when $s$ is even. We obtain that
$$M^0(h)=\int_{\gamma_h}\sum_{s=s_1}^{s_2}Q_{n,s}(x)\,
y^{s}\, dx=\int_{\gamma_h}\sum_{s=0}^{r}Q_{n,\ell+2s}(x)\, y^{\ell+2s}\, dx.$$ Moreover, since
$\int_{\gamma_h}\sin(jx)y^s\, dx$ is identically zero  on $(0,2)$  for all $j\ge 0$ by Green's
Theorem, we obtain
\begin{equation*}
    M^0(h)=\int_{\gamma_h}\sum_{s=0}^{r}Q_{n,\ell+2s}(x)
        \,y^{\ell+2s}\, dx=\int_{\gamma_h}\sum_{s=0}^{r}Q^e_{n,\ell+2s}(x)
        \,y^{\ell+2s}\, dx.
\end{equation*}
Now we prove item $(b_1)$. In view of Lemma \ref{L-LEM} and using the operator $\L$ defined in
\eqref{E-operatorL} we may write
\begin{align*}
    M^0(h)&=\int_{\gamma_h}\sum_{s=0}^r Q^e_{n,\ell+2s}(x)y^{\ell+2s}\, dx\\
                &=\int_{\gamma_h}\left(\sum_{s=0}^r \L^{n+r-s}
                    \left(Q^e_{n,\ell+2s}(x)\right)\right)y^{\ell+2n+2r}\, dx.
\end{align*}
From Lemma \ref{L-operatorL} (B) it follows that
\begin{equation*}
    \L^{n+r-s}\left(Q^e_{n,s}(x)\right)=\frac{R_{s}(\cos (x))}{\sin^{2(n+r-s)}(x)},
\end{equation*}
 for  certain polynomials $R_{s}$ of degree $n+r-s.$ Thus,  we obtain
\begin{multline*}
M^0(h) = \int_{\gamma_h}\left(\sum_{s=0}^r
                    \frac{R_{s}(\cos(x))}{\sin^{2(n+r-s)}(x)}\right)y^{\ell+2n+2r}\, dx \\
           = \int_{\gamma_h}\left(\sum_{s=0}^r
                    \frac{R_{s}(\cos(x))(\sin(x))^{2s}}{\sin^{2(n+r)}(x)}\right)y^{\ell+2n+2r}\, dx
             = \int_{\gamma_h}\frac{ R(\cos(x))}{\sin^{2(n+r)}(x)}y^{\ell+2n+2r}\, dx,
 \end{multline*}
 where $R(u)=\sum_{s=0}^{r} R_{s}(u)(1-u^2)^s$ is a polynomial of
 degree $n+2r.$ Since $2r<\ell+3$, the first part of  statement  $(b_1)$ follows from
 Lemma \ref{L-Tche}.

To prove the second part we need to show that for $r\le 2$ using the above procedure we can obtain
any prescribed polynomial $R(u)$ of degree $n+2r.$ For $r=0$ this follows because $ R(u)=
R_{0}(u)$, which is defined by
$$\L^n \left(Q^e_{n,1}(x)\right)=\frac
{R_{0}(\cos(x))}{\sin^{2n}(x)}.$$ Thus, for $i=0,\ldots,n$ choosing $Q_{n,0}(x)=\cos(ix)$ we
obtain $R(x)=P_{n,i}(x)$ which is a polynomial of degree exactly $n-i.$ Clearly the set
$\{P_{n,0},P_{n,1},\ldots,P_{n,n}\}$ is a basis of the polynomials of degree $n.$ This shows that
there exists a linear combination of perturbations $Q_{n,0}$ for which the corresponding Melnikov
function has exactly $n$ zeros. This proves the case $r=0.$ When $r=1$ we have that $R(u)=
R_{0}(u)+R_1(u)(1-u^2)$, where $R_0(u)$ and $R_1(u)$  are defined by
\begin{equation*}
    \L^{n+1}\left(Q^e_{n,0}(x)\right)=\frac{R_{0}(\cos(x))}{\sin^{2(n+1)}(x)} \mbox { and }
    \L^{n}\left(Q^e_{n,1}(x)\right)=\frac{R_{1}(\cos(x))}{\sin^{2n}(x)},
\end{equation*}
respectively. Choosing
$Q_{n,0}(x)=\frac{\cos((n-1)x)}{(2n-1)K(n-1)} $ and
$Q_{n,1}(x)=\frac{2n\cos(n x)}{K(n)} $ we get
$$R(x)=\frac{P_{n+1,n-1}(x)}{(2n-1)K(n-1)}+(1-x^2)2n\frac{P_{n,n}(x)}{K(n)}=(1+2nx^2)
+2n(1-x^2)=2n+1$$ which is a degree 0 polynomial. On the other hand choosing
$Q_{n,1}(x)=0$ and $Q_{n,0}(x)=\cos (ix)$ for $i=0,\ldots n,$ we obtain that $R$ is a polynomial of degree $i+1.$
Lastly, choosing $Q_{n,1}(x)=1$ and  $Q_{n,0}(x)=\cos (ix)$ we obtain that $R$ is a polynomial of
degree $n+2$. These choices give a basis of the polynomials of degree $n+2.$ The same type of
arguments and computations shows the result for $r=2$, but we omit these computations for the sake
of brevity.

Item  $(b_2)$ follows from the fact that in this case the Melnikov integral is
identically zero if and only if $Q_{n,s_1}$ depends only on $\sin (x)$, i.e.~$Q_{n,s_1}^e= 0$, in
which case the system is reversible and has a center at the origin for all $\ve.$
\end{proof}

\begin{proof}[Proof of Proposition~\ref{p-nova}]

The proof is conducted along the lines of the proof of Theorem~\ref{T-thm inside} and involves a lot of
computations. For the sake of brevity we only give details in the case $r=2.$ We need to study the
number of zeros of linear combinations of $I_{0,2s-1}(y)=\int_{\gamma_h} y^{2s-1}\,dx.$  By
Lemma~\ref{L-LEM}, using the notation of Theorem~\ref{T-THM} and the operator $\L$ given in
\eqref{E-operatorL} it holds that
\begin{align*}
I_{0,1}(y)&=\int_{\gamma_h} y\,dx= \int_{\gamma_h} \L[1](x)\, y^3\,dx= \int_{\gamma_h} \L^2[1](x)\,
 y^5\,dx,\\
I_{0,3}(y)&=\int_{\gamma_h} y^3\,dx= \int_{\gamma_h} \L[1](x)\, y^5\,dx,\\
I_{0,5}(y)&=\int_{\gamma_h} y^5\,dx.
\end{align*}
Simple computations give that
\[
\L[1](x)=-\frac{\cos(x)}{\sin^2(x)},\quad \L^2[1](x)=\frac{2\cos^2(x)+1}{\sin^4(x)}.
\]
These functions are even and well-defined in $(0,\pi).$
Notice that the three integrals $I_{0,s}, s=1,2,3$ all involve the term $y^5$. Therefore,  following
the notation of Theorem~\ref{T-THM} we have that $v=3$ and $n=3$.
Moreover, direct computations give that the Wronskians of the set of functions $(1,\L[1],\L^2[1])$
are
\begin{align*}
W[1]&=1,\quad W\big[1,\L[1]\big]=\frac{\cos^2(x)+1}{\sin^3(x)}\quad\mbox{and}\\
W\big[1,\L[1],\L^2[1]\big]&=\frac{4(\cos^6(x)+6\cos^4(x)+3\cos^2(x)+2)}{\sin^9(x)}.
\end{align*}
Clearly, each one of them does not vanish on $(0,\pi)$ and it holds that $n<v+2.$ Therefore, we
can apply Theorem~\ref{T-THM}, proving that the  functions $I_{0,1}, I_{0,3}$ and $I_{0,5}$ are an
ECT-system on $(0,2)$.
\end{proof}


\section{Proof of Theorem C}\label{pc}

To study the limit cycles in the rotary regions $\mathcal{R}^\pm$ we
resort to a result  of Gasull, Li and Torregrosa published in
\cite{Gasull2012b}. In this paper,  the authors introduce the family
of analytic functions
\begin{equation}
\label{E-intfam}
    J_{i,\alpha}(y)=\int_a^b{\frac{g^i(x)}{(1-yg(x))^{\alpha}}dx},\quad i=0,1,\ldots,n,
\end{equation}
where  $g$ is a continuous function, $a,b\in\R$ and $\alpha \in \R\setminus \Z^-$. These functions
are defined on the open interval $W$ where $1-yg(x)>0$ for all $x\in[a,b]$. They prove:
\begin{theorem}[Theorem A in \cite{Gasull2012b}]
\label{T-THM2} For any $n\in\N$ and any $\alpha \in  \R\setminus
\Z^-$, the ordered set of functions $(J_{0,\alpha},\dots,
J_{n,\alpha})$, as defined in \eqref{E-intfam}, is an ECT-system on
$W$.
\end{theorem}

The following proposition is a simple consequence of this result.

\begin{proposition}
\label{P-pre} For $n,p\in \N$ and $\beta=2p+1,$ the family
$$(I_{0,\beta}(h), I_{1,\beta}(h),\ldots, I_{n,\beta}(h)),$$ where the
functions $I_{i,\beta}(h)=\int_{0}^\pi\cos^i(x)y^{\beta/2}\, dx$ are
given in \eqref{E-integrals2}, is an ECT-system on $(2,\infty).$
Moreover,
 the same holds for the family \[(I^{(j)}_{0,\beta}(h),
I^{(j)}_{1,\beta}(h),\ldots, I^{(j)}_{n,\beta}(h))\] where, for any
$j>0$, $I^{(j)}_{i,\beta}$ denotes the $j^{th}$-derivative of
$I_{i,\beta}.$
\end{proposition}
\begin{proof}
 We have
$$I_{i,\beta}(h)=\int_{0}^{\pi}\cos^i(x)\left(h-1+\cos(x)\right)^{\frac
{\beta} 2} \, dx=(h-1)^\frac{\beta}2 J_{i,-\beta/2}\Big(\frac 1
{1-h}\Big)$$ where
$$J_{i,-\beta/2}(y)=\int_0^{\pi}\cos^i(x)\left (1-y\cos(x)\right)^{\frac{\beta}2} \, dx.$$
Therefore, it suffices to show that $(J_{0,-\beta/2}(y), J_{1,-\beta/2}(y),\ldots,
J_{n,-\beta/2}(y))$ is an ECT-system on $(-1,0),$ which is a direct application of Theorem
\ref{T-THM2}, choosing
 $g(x)=\cos(x)$ and $\alpha =-\beta/2$ in \eqref{E-intfam}.
Observing that for any $j>0$ we have
$$I^{(j)}_{i,\beta}(h)=\left(\frac{\beta}{2}\right)\left(\frac{\beta}{2}-1\right)\ldots
\left(\frac{\beta}{2}-(j-1)\right)I_{i,\beta-2j}(h)$$ completes the
proof.
\end{proof}


\begin{proof}[Proof of Theorem \ref{T-thm outside}]

We prove the result for $M^+(h)$, the proof for $M^-$ follows in the
same way. From Theorem~\ref{T-thm bounds} we have that
$$M^+(h)=\int_{\gamma_h^+}\left(Q^e_{n}(x)
y^{2p+1}+\sum_{s=0}^{r}Q^e_{n,s}(x)y^{2s}\right)\, dx,$$ and direct computations
imply that
$$\int_{\gamma_h^+}\sum_{s=0}^{r}Q^e_{n,s}(x)y^{2s}=Z_r(h)$$
 for a certain polynomial $Z_r$ of  degree $ r$. Furthermore, we have that
$$\int_{\gamma_h^+}Q^e_{n}(x)
y^{2p+1}\, dx=\sum_{i=0}^na_i\int_{\gamma_h^+}\cos(ix)y^{2p+1}\, dx
                    =\sum_{i=0}^nb_i\int_{\gamma_h^+}\cos^i(x)y^{2p+1}\, dx,
$$
for some constants $a_i,b_i\in\R$. So $M^+(h)$ belongs to the linear
space generated by $1,h,\ldots,h^r,I_{0,p}(h),\ldots,I_{n,p}(h)$
where $I_{i,p}(h)=\int_{0}^\pi\cos^i(x)(h-1+\cos(x))^{(2p+1)/2}\,
dx$ are given in \eqref{E-integrals2}. Therefore, it suffices to
show that the family
$$\left(1,h,\ldots,h^r,I_{0,p}(h),\ldots,I_{n,p}(h)\right)$$
is an ECT-system. To this end, let $k\le n$ and consider
$\varphi(h)=\sum_{i=0}^r a_i h^i+ \sum_{i=0}^kc_iI_{i,p}(h).$ Then
$\varphi^{(r+1)}(h)=\sum_{i=0}^kd_iI_{i,p+r+1}(h)$, and Proposition
\ref {P-pre} implies that  either $\varphi^{(r+1)}(h)$ is
identically zero or it has at most $k$ zeros counting multiplicity.
From Rolle's Theorem we obtain that either $\vp(h)$ is identically
zero or it has at most $k+r+1$ zeros counting multiplicity.

The proofs of the cases $Q_{n}(x)\equiv0$ or  $Q_{n,0}(x)\equiv
 Q_{n,1}(x)\equiv\cdots Q_{n,r}(x)\equiv0$ are much easier and
 follow by using the same arguments.
\end{proof}


\section{Simultaneous bifurcation of limit cycles}
 \label{SS-simultaneous}

The point here is to study the maximum number of limit cycles which may bifurcate simultaneously
in the entire cylinder from the periodic orbits of the integrable pendulum, i.e. in
$\mathcal{R}^0\cup\mathcal{R}^\pm$. Notice that this region corresponds  to all
$h\in(0,\infty)\setminus\{2\}$. To this end, we introduce the following notation: given a family
of systems of the form~\eqref{E-penduSysGen} we will say that it admits the configuration of limit
cycles $[c^-;c^0;c^+],$ where $c^-,c^0$ and $c^+$ are nonnegative integers, if there exist values
of the parameters of the system such that the three first order Melnikov integrals associated to
it, $M^-(h)$, $M^0(h)$ and $M^+(h)$ have $c^-$, $c^0$ and $c^+$ simple zeros, respectively, all
of them lying in the corresponding intervals of definition of the Melnikov functions,
that is, $(2,\infty),$ $(0,2)$
and $(2,\infty)$ respectively. With this notation, the results of Theorem~\ref{T-thm bounds} imply that
 the  configuration with the largest number of limit cycles, in case it is realizable,
 would be $[2n+2m+E(m/2)+2;
2n+2E((m-1)/2)+1;2n+2m+E(m/2)+2]$.\\

Even if each of the values of a configuration is optimal, to know when all maximal values are
attained simultaneously is a very intricate problem. In the results
of~\cite{Morozov1989,Sanders1986} for some subcases of system \eqref{E-penduSysGen} the maximal
values are not attained simultaneously, but it  may happen for similar systems, see for instance
\cite{Gar}. In this section we give some examples which illustrate that for other simple cases of
system \eqref{E-penduSysGen} the global optimal values are not attained simultaneously in the
three regions.
 We believe that this
general question is of interest and deserves further work.\\

Our first example is  the subfamily of systems of the form~\eqref{E-penduSysGen}, given by
\begin{equation}
      \label{E-ex1}
           \sys{y,}{-\sin( x)+ \ve (a_0+a_1\cos(x))y,}
     \end{equation}
with $(a_0,a_1)\in\R^2.$ From Theorems~\ref{T-thm inside} and~\ref{T-thm outside} we get that the configuration with the largest
number of limit cycles is $[1;1;1]$. Indeed, considering the functions
\begin{align*}
L_n^+(h)&=I_{n,1}(h)= \int_0^\pi \cos^n(x)\sqrt{h-1+\cos(x)}\,dx, \quad h\in(2,\infty),\\
L_n^0(h)&=I_{n,1}(h)= \int_0^{\arccos(1-h)} \cos^n(x)\sqrt{h-1+\cos(x)}\,dx, \quad h\in(0,2),
\end{align*}
see~\eqref{E-integrals2}, it holds that
\begin{align*}
M^{\pm} (h)&= \pm\big(2 a_0 L_0^+(h)+2a_1 L_1^+(h)\big),\quad h\in(2,\infty),\\
  M^{0} (h)&= 2 a_0 L_0^0(h)+2a_1 L_1^0(h),\quad h\in(0,2),
  \end{align*}
and from our analysis in the previous sections we know that $(L_0^+,L_1^+)$ and $(L_0^0, L_1^0)$
are ETC-systems for $h\in(2,\infty)$ and for $h\in(0,2),$ respectively. Notice that this implies
that the derivatives of the functions $Q^+:=L_1^+/L_0^+$ and $Q^0:=L_1^0/L_0^0$
do not vanish in their respective intervals of definition. In fact, it is easy to see that the
function
\[
Q(h)=\begin{cases} Q^0(h), & h\in(0,2],\\
                   Q^+(h), & h\in[2,\infty),\end{cases}
\]
defined for $h>0,$ is continuous, not differentiable at $h=2$ and decreasing.

Let us prove that the only two possible configurations for limit cycles of system~\eqref{E-ex1} are
$[1;0;1]$ and $[0;1;0]$.

It is clear that $c^+=c^-$ because $M^+(h)=-M^-(h).$ So, we only need to prove that $M^+$ and
$M^0$ can not  simultaneously have a zero in their respective intervals of definition. But this is a
straightforward consequence of the fact that $Q$ is globally decreasing.

Notice that the above  approach works for two integrals due to the nice analogy between the
non-vanishing Wronskians and the monotonicity of the quotients. The generalization to an arbitrary
number of integrals however is far from obvious.

As a second example, consider   the subfamily of systems of the form~\eqref{E-penduSysGen}, given by
\begin{equation}
      \label{E-ex2}
           \sys{y,}{-\sin( x)+ \ve \big(a_0+a_1\cos^2(x)+ a_2 y^{2r+1}\big),}
     \end{equation}
with $(a_0,a_1,a_2)\in\R^3$ and $r\in\N.$ Again, from Theorems~\ref{T-thm inside} and~\ref{T-thm outside} we see that the configuration
with the largest number of  limit cycles possible is $[2;0;2]$. This is because
\begin{align*}
M^{\pm} (h)&=  \alpha_0+\alpha_1 h \pm 2 a_2 W_{2r+1}^+(h)  ,\quad h\in(2,\infty),\\
  M^{0} (h)&=2 a_2 W_{2r+1}^0(h) ,\quad h\in(0,2),
  \end{align*}
for some linearly independent parameters $(\alpha_0,\alpha_1)$, where
\begin{align*}
W_{2r+1}^+(h)&=I_{0,2r+1}^+(h)= \int_0^\pi \big(\sqrt{h-1+\cos(x)}\big)^{2r+1}\,dx,
\quad h\in(2,\infty),\\
W_{2r+1}^0(h)&=I_{0,2r+1}^0(h)= \int_0^{\arccos(1-h)}\big(\sqrt{h-1+\cos(x)}\big)^{2r+1}\,dx, \quad
h\in(0,2),
\end{align*} see~\eqref{E-integrals2}. Hence, since  we know that $(1,h,W_{2r+1}^+(h))$ and
$W_{2r+1}^0(h)$ are ETC-systems for $h\in(2,\infty)$ and for $h\in(0,2),$ respectively, we get
that $c^0=0$, because $W_{2r+1}^0$ does not vanish on $(0,2),$ and that $c^\pm\leq 2$ . Let us
prove that the value $2$ cannot be attained simultaneously by both $c^+$ and $c^-$. Indeed,
when $a_2=0$ the result is trivial and $c^+=c^-\le1.$ When $a_2\ne0$  our problem is equivalent to
finding the maximum number of zeros in $(2,\infty)$ for each of the equations
\[
g_r(h)=\beta_0+\beta_1 h\quad \mbox{and}\quad g_r(h)=-\beta_0-\beta_1 h,
\]
where  $(\beta_0,\beta_1)\in\R^2$ and  $g_r(h)=W_{2r+1}^+(h)$. It is clear that in the interval
$(2,\infty)$ it holds that
\[
g_r(h)>0,\quad  g'_r(h)>0 \quad\mbox{and}\quad g''_r(h)\begin{cases} <0, \quad\mbox{when}\quad
r=0,\\>0,\quad\mbox{when}\quad r\ge1.\end{cases}
\]
From the above inequalities it is not difficult to prove that when $r=0$ the realizable
configurations with a maximal number of limit cycles for system~\eqref{E-ex2} are $[1;0;1],$ $[2;0;0]$
or
$[0;0;2]$. When $r\ge1$ these configurations are $[2;0;1]$ or $[1;0;2].$



\subsection*{Acknowledgments}

The first and third authors are supported by  MINECO grants, with respective  numbers
MTM2013-40998-P and MTM2014-52209-C2-1-P. The first  author is also supported by a CIRIT grant
number 2014SGR568. The second author is  supported by the
 project J3452 ``Dynamical Systems Methods in Hydrodynamics'' of the Austrian Science Fund (FWF).

\end{document}